\documentclass[12pt,reqno]{amsart}
\usepackage{amsfonts}
\usepackage{tikz}
\usetikzlibrary{arrows.meta, positioning}
\usepackage{a4wide,graphicx}
\usepackage{amsmath,amsthm,amssymb,amscd}
\usepackage{latexsym}
\usepackage{hyperref}
\usepackage[numbers,sort&compress]{natbib}
\usepackage{hypernat}
\usepackage{color}
\usepackage{subfigure}
\usepackage{pdflscape}

\allowdisplaybreaks
\numberwithin{equation}{section}

\makeatletter
\newcommand{\rmnum}[1]{\romannumeral #1}
\newcommand{\Rmnum}[1]{\expandafter\@slowromancap\romannumeral #1@}
\makeatother

\newtheorem{theorem}{Theorem}[section]
\newtheorem{proposition}[theorem]{Proposition}

\newtheorem{lemma}[theorem]{Lemma}

\theoremstyle{definition}

\newcommand{\va}{\varepsilon}

\newcommand{\la}{\lambda}

\newcommand{\be}{\beta}
\newcommand{\al}{\alpha}

\newcommand{\mc}{\mathcal}
\newcommand{\intg}{\int_{\mc{G}}}
\newcommand*{\dif}{\,\mathrm{d}}
\newcommand{\rhp}{\rightharpoonup}

\def\r{\mathbb{R}}

\begin{document}
	\title [Multiple positive bound states for NLS on noncompact metric graphs]
	{Multiple positive bound states for NLS equations on noncompact metric graphs with potential
	}
	
	\author{
		Quan Liu \textsuperscript{1}$^{, \ast}$
	}
	
	\thanks{
		\textsuperscript{1} Department of Mathematics, Gannan Normal University, Ganzhou, Jiangxi, 341000, P.R. China
	}
	\thanks{Email: quan\_liu@gnnu.edu.cn}
	\thanks{*Corresponding author}

\begin{abstract}
	
	In this paper, we establish the existence of geometrically distinct bound states for the subcritical Nonlinear Schr\"{o}dinger (NLS) equation on a noncompact metric graph, provided that the mass is sufficiently large. By bound states we mean critical points of the NLS energy functional constrained to the manifold of fixed $L^2$-norm.
	We demonstrate that the NLS equation admits at least as many positive bound states as the number of bounded edges of the graph. Our approach is purely variational and relies on a doubly constrained minimization procedure. 
	It is worth noting that, compared with the case where the potential vanishes, the scaling structure of solitons is destroyed here. Moreover, the classical energy asymptotic behavior $\mathcal{E}(\phi_{\mu},\mathbb{R})=-\theta \mu^{2\beta+1}$ is no longer preserved.
	
	{\bf Key words}: variational methods; Nonlinear Schr\"{o}dinger equations; noncompact metric graph; multiple solutions.
	
	{\bf AMS Subject Classifications:} 35R02, 35Q65, 49J40
\end{abstract}

\maketitle

\section{Introduction}
In recent decades, nonlinear Schr\"{o}dinger equations (NLS) on metric graphs has been intensively studied. 
The NLS equation on graphs can describe the evolutionary phenomena in the networks and it holds significant importance in engineering and scientific research, such as waveguides, photonic crystals and nanostructures, to name a few. 
For more physical background of such equations we refer to surveys \cite{Noja2014,BK,BCFK}.
In this paper, we investigate the existence of bounded states for the following NLS
\begin{align}\label{eq1}
\left\{\begin{array}{ll}
-u^{\prime\prime}+(W(x)+\la)u=|u|^{p-2}u, \quad 2<p<6, \\
\|u\|^2_{L^2(\mc{G})}=\mu
\end{array} \right.
\end{align}
on a noncompact metric graph $\mc{G}$ together with the continuity and Kirchhoff conditions (see \eqref{eq6} below) at every vertex $\rm{v}$ of $\mc{G}$, where $\mu>0$ is given in advance and the unknown parameter $\la$ is a Lagrange multiplier. 
The solutions of \eqref{eq1} give the standing waves of the following time-dependent NLS equation on $\r\times\mc{G}$:
\begin{equation*}
	\mathrm{i}\dfrac{\partial\psi(t,x)}{\partial t}=-\dfrac{\partial^2\psi(t,x)}{\partial x^2}+W(x)\psi(t,x)-|\psi(t,x)|^{p-2}\psi(t,x),
\end{equation*}
where $\psi(t,x)=\exp(\mathrm{i}\lambda t)u(x)$.
The constraints in \eqref{eq1} are physically meaningful and can be used to represent the conservation of mass (or charge) and energy.
For these reasons, the existence of solutions to \eqref{eq1} on metric graphs has attracted a lot of attention in the last decade, mainly in the autonomous case, which respectively correspond to $W(x)=0$. Within this framework, one is concerned with the existence of ground state solutions, that is, global minimizers of the energy under the mass constraint (see \cite{DDGS2026,ASTcmp,AST2015,AST2016,PS2022} for noncompact $\mc{G}$, and \cite{JS2026,CDS,D2018JDE,BMD2021,CJS2022} for compact $\mc{G}$). It is worth noting that in \cite{AST2019}, the existence of multiple local minimizers of the energy functional under a given mass constraint was proved, and the number of solutions is related to the topological properties of the non-compact metric graph.
When $\lambda$ is fixed and the mass is unknown, we refer to \cite{ACT2024,DDGS2023,P2018}.
For a review results and open problems related to the NLS equation on metric graphs we refer to \cite{Noja2014,DST2020}.

In what follows we shall focus attention on the existence of the bounded state. Roughly speaking, we shall prove that there exist at least as many bounded states as the number of bounded edges of $\mc{G}$ when $\mc{G}$ is any noncompact metric graph and $W$ is a bounded potential.
Our approach is variational and based on a doubly constrained minimization procedure. 
Firstly, it is standard to check that the nontrivial critical points of $\mc{E}(\cdot,\mc{G})$ (see \eqref{eq3} for a precise definition) correspond exactly to the weak solutions in $\mathbb{H}$ (see \eqref{Neqeq1} for a precise definition) of Eq. \eqref{eq1}. 
Secondly, for each bounded edge $e$ of $\mc{G}$, we minimize $\mc{E}(u,\mc{G})$ among all functions which $L^\infty$ norm is achieved on the edge $e$ and satisfies mass constraint (see \eqref{eq8}). 
In order to obtain a minimizer under double constraints, we establish a concentration compactness lemma on noncompact metric graphs in the case where the potential is non-vanishing.
It is worth noting that the set $\mc{H}_\mu(\mc{G})$ which we seek the minimizer is not a compact set since the double constraints. Thus, the existence of a minimizer is nontrivial. 
Then, we devote to proving the concentration compactness principle for minimizing sequence. 
We point out that, to establish this result, the techniques based on translation, usually utilized in the Euclidean setting, do not work, since $\mc{G}$ is not translation invariant. 
Fortunately, we can show that the limit energy level along the halfline of $\mc{G}$ is not smaller than the energy level of a soliton. Then we show that minimizing sequences are bounded, strongly compact in $L^2$ and in fact convergent if its mass is large enough.
However, this not implies that the minimizer $u$ is a critical point of the NLS energy under the mass constraint. In fact, when $u$ moves in $S_{\mu}(\mc{G})$ ($L^2-$sphere in $\mathbb{H}$), it may violate the additional constraint condition that $u$ achieves its maximum value on the edges $e$ of graph $\mc{G}$. Finally, we proved that $u$ obtains its maximum value only on $e$ when $\mu$ is large enough. 
Thus, the minimizer $u$ lies in the interior of the additional constraint. Therefore, we can prove that the minimizer $u$ is a critical point of the NLS energy under the mass constraint. It is worth noting that, compared with the case where the potential vanishes, the scaling structure of solitons is destroyed here. Moreover, the classical energy asymptotic behavior $\mc{E}(\phi_{\mu},\r)=-\theta \mu^{2\be+1}$ is no longer preserved.

We point out that the solutions obtained in this paper are not, in general, ground state solutions. Moreover, if the graph $\mc{G}$ has $k$ bounded edges and $\mu$ is large enough, we obtain at least $k$ geometrically distinct solutions. For example, the following graph (see Figure \ref{fig:0}) consists of $3$ unbounded edges, $11$ bounded edges and $7$ vertices. Therefore, we claim that the Eq.\eqref{eq1} at least $11$ bounded states when $\mu$ large enough. 

\begin{flushleft}
	{\bf Basic notation and main result}
\end{flushleft}

Firstly, we briefly review some basic notations (see \cite{BCFK} or \cite{EKKST} for more details). We recall that a graph is a couple $\mc{G}:=(V,E)$, where $V=\{\mathrm{v}_k\}$ is a set of vertices and $E=\{e_j \}$ is a set of edges. We always assume that the cardinalities $|E|$ of $E$ and $|V|$ of $V$ are finite.
A metric graph $\mc{G}$ is a graph with a metric structure on any edge. Every bounded edge of $\mc{G}$ is identified with a oriented segment, i.e. $e\sim I_e:=[0,l_e]$, where $e\in E$ and $ l_e$ denotes the length of $e$, while any unbounded edge is identified with a closed half-line, i.e. $I_e=[0,+\infty)$. If $x,y\in \mc{G} $, the distance $d(x, y)$ is the infimum of the length of the paths connecting the two points. Thus, $(\mc{G},d)$ is a metric space ($\mc{G}$, for short). The metric graph $\mc{G}$ is a noncompact metric graph if and only if there exists $e_0\in E$ such that $l_{e_0}=+\infty$. Two very special cases are $\mc{G}$ consists of just one unbounded edge (i.e. $\mc{G}=\r^+$) and $\mc{G}$ obtained by gluing together two unbounded edges (i.e. $\mc{G}=\r$). According to the topology of graph, the metric graph allows for self-loops and unbounded edges (for example, see Figure \eqref{fig:0}). 

\begin{figure}[hb]
	\centering
	\begin{tikzpicture}[x=0.75pt,y=0.75pt,yscale=-1,xscale=1]
	\draw  (81,240) .. controls (81,226.19) and (92.19,215) .. (106,215) .. controls (119.81,215) and (131,226.19) .. (131,240) .. controls (131,253.81) and (119.81,265) .. (106,265) .. controls (92.19,265) and (81,253.81) .. (81,240) -- cycle ;
	\draw [color={rgb, 255:red, 0; green, 0; blue, 0 }  ,draw opacity=1 ]   (212.33,241.47) -- (131,240) ;
	\draw [shift={(131,240)}, rotate = 181.03] [color={rgb, 255:red, 0; green, 0; blue, 0 }  ,draw opacity=1 ][fill={rgb, 255:red, 0; green, 0; blue, 0 }  ,fill opacity=1 ][line width=0.75] (0, 0) circle [x radius= 3.35, y radius= 3.35];
	\draw [shift={(212.33,241.47)}, rotate = 181.03] [color={rgb, 255:red, 0; green, 0; blue, 0 }  ,draw opacity=1 ][fill={rgb, 255:red, 0; green, 0; blue, 0 }  ,fill opacity=1 ][line width=0.75]      (0, 0) circle [x radius= 3.35, y radius= 3.35];
	\draw [color={rgb, 255:red, 0; green, 0; blue, 0 }  ,draw opacity=1 ]   (290.33,240.47) -- (212.33,241.47) ;
	\draw [shift={(212.33,241.47)}, rotate = 179.27] [color={rgb, 255:red, 0; green, 0; blue, 0 }  ,draw opacity=1 ][fill={rgb, 255:red, 0; green, 0; blue, 0 }  ,fill opacity=1 ][line width=0.75]      (0, 0) circle [x radius= 3.35, y radius= 3.35]   ;
	\draw [shift={(290.33,240.47)}, rotate = 179.27] [color={rgb, 255:red, 0; green, 0; blue, 0 }  ,draw opacity=1 ][fill={rgb, 255:red, 0; green, 0; blue, 0 }  ,fill opacity=1 ][line width=0.75]      (0, 0) circle [x radius= 3.35, y radius= 3.35];
	\draw [color={rgb, 255:red, 0; green, 0; blue, 0 }  ,draw opacity=1 ]   (212.33,190.47) -- (212.33,241.47) ;
	\draw [shift={(212.33,241.47)}, rotate = 90] [color={rgb, 255:red, 0; green, 0; blue, 0 }  ,draw opacity=1 ][fill={rgb, 255:red, 0; green, 0; blue, 0 }  ,fill opacity=1 ][line width=0.75]      (0, 0) circle [x radius= 3.35, y radius= 3.35]   ;
	\draw [shift={(212.33,190.47)}, rotate = 90] [color={rgb, 255:red, 0; green, 0; blue, 0 }  ,draw opacity=1 ][fill={rgb, 255:red, 0; green, 0; blue, 0 }  ,fill opacity=1 ][line width=0.75]      (0, 0) circle [x radius= 3.35, y radius= 3.35]   ;
	\draw [color={rgb, 255:red, 0; green, 0; blue, 0 }  ,draw opacity=1 ]   (212.33,320.47) -- (212.33,241.47) ;
	\draw [shift={(212.33,241.47)}, rotate = 270] [color={rgb, 255:red, 0; green, 0; blue, 0 }  ,draw opacity=1 ][fill={rgb, 255:red, 0; green, 0; blue, 0 }  ,fill opacity=1 ][line width=0.75]      (0, 0) circle [x radius= 3.35, y radius= 3.35];
	\draw [shift={(212.33,320.47)}, rotate = 270] [color={rgb, 255:red, 0; green, 0; blue, 0 }  ,draw opacity=1 ][fill={rgb, 255:red, 0; green, 0; blue, 0 }  ,fill opacity=1 ][line width=0.75]      (0, 0) circle [x radius= 3.35, y radius= 3.35]   ;
	\draw [color={rgb, 255:red, 0; green, 0; blue, 0 }  ,draw opacity=1 ]   (290.33,319.47) -- (290.33,240.47) ;
	\draw [shift={(290.33,240.47)}, rotate = 270] [color={rgb, 255:red, 0; green, 0; blue, 0 }  ,draw opacity=1 ][fill={rgb, 255:red, 0; green, 0; blue, 0 }  ,fill opacity=1 ][line width=0.75]      (0, 0) circle [x radius= 3.35, y radius= 3.35]   ;
	\draw [shift={(290.33,319.47)}, rotate = 270] [color={rgb, 255:red, 0; green, 0; blue, 0 }  ,draw opacity=1 ][fill={rgb, 255:red, 0; green, 0; blue, 0 }  ,fill opacity=1 ][line width=0.75]      (0, 0) circle [x radius= 3.35, y radius= 3.35]   ;
	\draw [color={rgb, 255:red, 0; green, 0; blue, 0 }  ,draw opacity=1 ]   (290.33,319.47) -- (212.33,320.47) ;
	\draw [shift={(212.33,320.47)}, rotate = 179.27] [color={rgb, 255:red, 0; green, 0; blue, 0 }  ,draw opacity=1 ][fill={rgb, 255:red, 0; green, 0; blue, 0 }  ,fill opacity=1 ][line width=0.75]      (0, 0) circle [x radius= 3.35, y radius= 3.35]   ;
	\draw [shift={(290.33,319.47)}, rotate = 179.27] [color={rgb, 255:red, 0; green, 0; blue, 0 },draw opacity=1 ][fill={rgb, 255:red, 0; green, 0; blue, 0 } ,fill opacity=1 ][line width=0.75] (0, 0) circle [x radius= 3.35, y radius= 3.35];
	\draw [color={rgb, 255:red, 0; green, 0; blue, 0 }  ,draw opacity=1 ]   (382.33,239.47) -- (290.33,240.47) ;
	\draw [shift={(290.33,240.47)}, rotate = 179.38] [color={rgb, 255:red, 0; green, 0; blue, 0 }  ,draw opacity=1 ][fill={rgb, 255:red, 0; green, 0; blue, 0 }  ,fill opacity=1 ][line width=0.75](0, 0) circle [x radius= 3.35, y radius= 3.35];
	\draw [shift={(382.33,239.47)}, rotate = 179.38] [color={rgb, 255:red, 0; green, 0; blue, 0 }  ,draw opacity=1 ][fill={rgb, 255:red, 0; green, 0; blue, 0 }  ,fill opacity=1 ][line width=0.75]      (0, 0) circle [x radius= 3.35, y radius= 3.35]   ;
	\draw [color={rgb, 255:red, 0; green, 0; blue, 0 }  ,draw opacity=1 ]   (520.33,240.47) -- (382.33,239.47) ;
	\draw    (212.33,241.47) -- (290.33,319.47) ;
	\draw [color={rgb, 255:red, 0; green, 0; blue, 0 }  ,draw opacity=1 ]   (212.33,320.47) -- (74.33,319.47) ;
	\draw  [draw opacity=0] (288.99,241.43) .. controls (289.7,241.37) and (290.42,241.33) .. (291.15,241.33) .. controls (308.73,241.3) and (323.02,258.61) .. (323.06,280) .. controls (323.11,301.39) and (308.89,318.76) .. (291.31,318.8) .. controls (289.83,318.8) and (288.38,318.69) .. (286.96,318.45) -- (291.23,280.07) -- cycle ; \draw   (288.99,241.43) .. controls (289.7,241.37) and (290.42,241.33) .. (291.15,241.33) .. controls (308.73,241.3) and (323.02,258.61) .. (323.06,280) .. controls (323.11,301.39) and (308.89,318.76) .. (291.31,318.8) .. controls (289.83,318.8) and (288.38,318.69) .. (286.96,318.45) ;  
	\draw  [draw opacity=0] (290.31,320.06) .. controls (290.34,320.79) and (290.35,321.51) .. (290.32,322.24) .. controls (289.67,339.81) and (271.7,353.13) .. (250.19,352) .. controls (228.68,350.87) and (211.77,335.7) .. (212.42,318.13) .. controls (212.47,316.65) and (212.65,315.19) .. (212.95,313.78) -- (251.37,320.19) -- cycle ; \draw   (290.31,320.06) .. controls (290.34,320.79) and (290.35,321.51) .. (290.32,322.24) .. controls (289.67,339.81) and (271.7,353.13) .. (250.19,352) .. controls (228.68,350.87) and (211.77,335.7) .. (212.42,318.13) .. controls (212.47,316.65) and (212.65,315.19) .. (212.95,313.78) ;  
	\draw [color={rgb, 255:red, 0; green, 0; blue, 0 }  ,draw opacity=1 ]   (518.33,320.47) -- (290.31,320.06);
	\draw (522,233) node [anchor=north west][inner sep=0.75pt] {$\infty $};
	\draw (131,229.4) node [anchor=north west][inner sep=0.75pt]    {$\mathrm{v}_{1}$};
	\draw (218,229.4) node [anchor=north west][inner sep=0.75pt]    {$\mathrm{v}_{2}$};
	\draw (218,182.4) node [anchor=north west][inner sep=0.75pt]    {$\mathrm{v}_{5}$};
	\draw (296,229.4) node [anchor=north west][inner sep=0.75pt]    {$\mathrm{v}_{3}$};
	\draw (388,229.4) node [anchor=north west][inner sep=0.75pt]    {$\mathrm{v}_{4}$};
	\draw (50,309.4) node [anchor=north west][inner sep=0.75pt]{$\infty $};
	\draw (525,311.4) node [anchor=north west][inner sep=0.75pt]    {$\infty $};
	\draw (218,308.4) node [anchor=north west][inner sep=0.75pt]    {$\mathrm{v}_{6}$};
	\draw (299,323.4) node [anchor=north west][inner sep=0.75pt]    {$\mathrm{v}_{7}$};
	\end{tikzpicture}
	\caption{A graph with 7 vertices and 14 edges: 3 unbounded edges (halflines) and 11 bounded edges.}
	\label{fig:0}
\end{figure}

A function $u$ on the metric graph is a family of functions $u_e:I_e\to \r$, i.e. $u\equiv\{u_e\}_{e\in E}$. If a function $u: \mc{G}\to \r$ is integrable, then each component $u_e$ of $u$ on the edge $e$ is integrable on the $I_e$ and 
\[\int_{\mc{G}}u\dif x=\sum_{e\in E}\int_{I_e}u_e(x)\dif x.\]
Similarly, one can define $L^p$ spaces over $\mc{G}$, denote by $L^p(\mc{G})$, with norm
\[\|u\|_{L^p(\mc{G})} =\left( \sum_{e\in E}\int_{I_e}|u_e(x)|^p\dif x\right)^{\frac{1}{p}}, \qquad 1\leq p<+\infty\]
and 
\[ \|u\|_{L^{\infty}(\mc{G})}=\max_{e\in E}\|u_e\|_{L^\infty(I_e)}=\max_{e\in E}\sup_{x\in I_e}|u_e(x)|, \qquad  p=+\infty. \]
In particular, $L^2(\mc{G})$ is a Hilbert space and its inner product defined by 
\[\left(u,v \right):=\int_{\mc{G}} uv\dif x,  \]
such that $\|u\|_{L^2(\mc{G})} =\sqrt{\left(u,u \right)}.$ Clearly, continue function space over $\mc{G}$, denote by $C(\mc{G},\r)$, is naturally defined edge by edge. 
Furthermore, the Sobolev space $H^1(\mc{G})$ is defined in the nature way
\[H^1(\mc{G}):=\left\lbrace u\in C(\mc{G},\r):\quad u_e\in H^1(I_e) \;\;\text{for all}\;\; e\in E \right\rbrace \]
and the norm in $H^1(\mc{G})$ is defined as 
\[\|u\|_{H^1(\mc{G})}:=\left( \sum_{e\in E}\int_{I_e}|u_e^\prime |^2+|u_e|^2\dif x  \right) ^{\frac{1}{2}}.\]
It is noted that the continuity condition in the definition of $H^1(\mc{G})$ means that if $\mathrm{v}$ belongs to different edges $e_k$, the corresponding function $u_{e_k}$ takes the same value at $\mathrm{v}$.

In order to states our results, we give the assumptions on $W$. First, we assume that $W\in C(\mc{G},\r)$ satisfies the global essential supremum is actually attained at infinity and is the same on all halflines. More precisely, we assume that 
\begin{itemize}
	\item[$(W_1)$] $W\in C(\mc{G},\r)$ and $\lim\limits_{x\to\infty}W_e(x)=\|W\|_{L^{\infty}(\mc{G})}$ for all unbounded edges $e$.
\end{itemize}
Moreover, we will assume that $W$ satisfies the following condition:
\begin{itemize}
	\item[$(W_2)$] there exists $W^+$ and $W_->0$ such that $0<W_-\leq W(x)\leq W^{+}<\infty$ for all $x\in \mc{G}$.
\end{itemize}

We shall work in space 
\begin{equation}\label{Neqeq1}
	\mathbb{H}:=\left\lbrace u\in H^1(\mc{G}): \;\;\sum_{e\in E}\int_{I_e}W_e|u_e|^2\dif x<\infty \right\rbrace  
\end{equation}
with norm 
\[\|u\|_{\mathbb{H}}:=\left( \sum_{e\in E}\int_{I_e}|u_e^\prime |^2+W_e|u_e|^2\dif x  \right) ^{\frac{1}{2}}.\]

We define the energy functional $\mc{E}(\cdot,\mc{G}): \mathbb{H}\to \r$ by 
\begin{align}\label{eq3}
	\mc{E}(u,\mc{G})&:=\dfrac{1}{2}\intg |u^\prime (x)|^2+W(x)|u|^2 \dif x-\dfrac{1}{p}\intg |u(x)|^p \dif x.
\end{align} 
It is well know that if $u_0\in \mathbb{H}$ is a critical point of $\mc{E}(\cdot,\mc{G})$ submitted to the mass constraint
\begin{equation*}
	S_{\mu}(\mc{G}):=\left\lbrace u\in \mathbb{H}:\;\;\|u\|^2_{L^2(\mc{G})}=\mu \right\rbrace,
\end{equation*}
then there is a Lagrange multiplier $\la_0$ such that $(u_0,\la_0)$ solves \eqref{eq1}. More precisely, if $u_0\in \mathbb{H}$ is such a critical point, then there exists a Lagrange multiplier $\la_0\in \r$ such that $(u_0,\la_0)$ satisfies the following equation:
\begin{align}\label{eq6}
\left\{\begin{array}{ll}
-u^{\prime\prime}+(W(x)+\la)u=|u|^{p-2}u\quad \text{for every edge} \;\; e\in E,\\
\sum\limits_{e\succ \mathrm{v}}\dfrac{\dif u_e}{\dif x}(\mathrm{v})=0\quad \text{at every vertex} \;\;\mathrm{v}\in V,
\end{array} \right.
\end{align}
where $e\succ \mathrm{v}$ denotes the edge $e$ such that $\mathrm{v}$ is an endpoint of $e$. The second equation of \eqref{eq6} is the so-called Kirchhoff condition and is a natural from of continuity of $u^\prime$ at the vertices of $\mc{G}$ (see \cite{AST2015,ACFN2014,ACFN2016} for more details).

Our main existence results are as follows:
\begin{theorem}\label{thm1}
	Let $\mc{G}$ be a noncompact and connected metric graph, having $k$ bounded edges $\{e_i\}_{i=1}^k$. If $W$ satisfies $(W_1)$ and $(W_2)$ then equation \eqref{eq1} admits at least $k$ bound states $\{u_i\}_{i=1}^k$ for large $\mu$. Moreover, the maximum of $u_i$ is attained only at the edge $e_i$ and either $u_i>0$ or $u_i<0$.
\end{theorem}

This paper is organized as follows. In Section 2, we present some preliminaries. In Section 3, we show the proof of Theorem \ref{thm1}.

\section{Preliminaries}
Let 
\begin{equation*}
	\mc{H}(\mc{G}):=\{u\in \mathbb{H}: \|u\|_{L^\infty(e)}=\|u\|_{L^\infty(\mc{G})} \} 
\end{equation*}
and 
\begin{equation*}
	\mc{H}_\mu(\mc{G}):=\mc{H}(\mc{G})\cap S_{\mu}(\mc{G}),
\end{equation*}
where $e$ is a bounded edge of $\mc{G}$.
We will investigate the following minimization problem: find a function $u\in \mc{H}_\mu(\mc{G})$ such that 
\begin{equation}\label{eq8}
\mc{E}(u,\mc{G})=\inf_{v\in \mc{H}_\mu(\mc{G})}\mc{E}(v,\mc{G}).
\end{equation}
We first give two basic tools in normalization problems, i.e. Gagliardo-Nirenberg inequality and $L^\infty$ estimate.
\begin{lemma}\label{lem1}
	Let $\mc{G}$ is a noncompact metric graph and $p\in (2,6)$. Then for any $u\in H^1(\mc{G})$, there exists $C:=C(p)$ such that 
	\begin{equation}\label{eq9}
		\|u\|^p_{L^p(\mc{G})}\leq C\|u\|^{\frac{p}{2}+1}_{L^2(\mc{G})}\cdot \|u^\prime\|^{\frac{p}{2}-1}_{L^2(\mc{G})}, 
	\end{equation}
	and 
	\begin{equation}\label{eq10}
		\|u\|^2_{L^\infty(\mc{G})}\leq 2\|u\|_{L^2(\mc{G})}\cdot \|u^\prime\|_{L^2(\mc{G})}.
	\end{equation}
\end{lemma}
Note that inequality \eqref{eq9} cannot hold if $\mc{G}$ is a compact graph (take a constant function as a counterexample). We refer to Proposition 4.1 in \cite{Ten2016} and Proposition 2.1 in \cite{AST2016} for detailed proofs.

Next, we recall the definition of \textit{soliton}. Clearly, the energy functional \eqref{eq3} can be obtained by adding the influence term describing the external potential $W$ via the standard NLS energy functional 
\begin{equation}\label{eq5}
\mc{E}_{\mathrm{NLS}}(u,\mc{G}):=\dfrac{1}{2}\intg |u^\prime (x)|^2\dif x-\dfrac{1}{p}\intg |u(x)|^p \dif x.
\end{equation}
If $\mc{G}=\r$, then the ground states of \eqref{eq5} under the mass constrain is called \textit{soliton}. 
We denoted by $\phi_{\mu}$ the soliton
\[\phi_{\mu}(x)=\mu^{\alpha}C_p \text{sech}^{\frac{\alpha}{\beta}}\left(C^\prime_p \mu^{\beta}x\right), \]
where $C_p$ and $C^\prime_p $ are positive constants and 
 $$\al=\frac{2}{6-p}>0,\qquad \be=\frac{p-2}{6-p}>0.$$
The function $\phi_{\mu}$ is the unique (up to translations) positive ground state of $\mc{E}_{\mathrm{NLS}}(u,\r)$ at mass $\mu$ (see \cite{BL}).
Similarly, when $\mc{G}=\r^+$, the unique ground state of \eqref{eq5} under the mass constrain is called ``\textit{half soliton}" (i.e. a soliton of mass $2\mu$ on $\r$) and we refer to \cite{Noja2014} for more details.


\begin{lemma}\label{lem2}
	Fix $2<p<6$. Then there exists a constant $\theta>0$ which depend on only $p$ such that 
	\begin{equation}\label{eq11}
		\min_{u\in S_{\mu}(\r)}\mc{E}_{\mathrm{NLS}}(u,\r)=\mc{E}_{\mathrm{NLS}}(\phi_{\mu},\r)=-\theta \mu^{2\be+1}, 
	\end{equation}
	and 
	\begin{equation}\label{eq12}
		\min_{u\in S_{\mu}(\r^+)}\mc{E}_{\mathrm{NLS}}(u,\r^+)=\dfrac{1}{2}\mc{E}_{\mathrm{NLS}}(\phi_{2\mu},\r)=-\theta 2^{2\be}\mu^{2\be+1}.
	\end{equation}
	Moreover, for any noncompact metric graph $\mc{G}$, it was proved in \cite[Theorem 2.2]{AST2015} that 
	\begin{equation}\label{eq13}
		-\theta 2^{2\be}\mu^{2\be+1}\leq \inf_{u\in S_\mu(\mc{G})}\mc{E}_{\mathrm{NLS}}(u,\mc{G})\leq -\theta \mu^{2\be+1}.
	\end{equation}
\end{lemma}
\begin{lemma}\label{lem3}
	$\inf_{v\in \mc{H}_\mu(\mc{G})}\mc{E}(v,\mc{G})$ is well defined. Moreover, if $\{u_n\}\subset \mc{H}_\mu(\mc{G})$ is a minimizing sequence of $\mc{E}(v,\mc{G})$, then there exist a subsequence of $\{u_n\}$ (still denote $\{u_n\}$) and $u$ such that $u_n\rhp u$ in $\mc{H}(\mc{G})$ as $n\to\infty$.
\end{lemma}
\begin{proof}
	Fix $u\in \mc{H}_\mu(\mc{G})\subset S_{\mu}(\mc{G})$. According to $(W_2)$ and \eqref{eq13} of Lemma \ref{lem2}, we have 
	\begin{align*}
		\mc{E}(u,\mc{G})&=\dfrac{1}{2}\intg |u^\prime (x)|^2+W(x)|u|^2 \dif x-\dfrac{1}{p}\intg |u(x)|^p \dif x\\
		&\geq \dfrac{1}{2}\intg |u^\prime (x)|^2\dif x+\dfrac{W_{-}}{2}\mu-\dfrac{1}{p}\intg |u(x)|^p \dif x\\
		&=\mc{E}_{\mathrm{NLS}}(u,\mc{G})+\dfrac{W_{-}}{2}\mu\\
		&\geq-\theta 2^{2\be}\mu^{2\be+1}+\dfrac{W_{-}}{2}\mu\\
		&>-\infty.
	\end{align*}
	Let $\{u_n\}\subset \mc{H}_\mu(\mc{G})$ is a minimizing sequence of $\mc{E}(v,\mc{G})$. By \eqref{eq9} and \eqref{eq10} of Lemma \ref{lem1}, we deduce that
	\begin{align*}
		\mc{E}(u_n,\mc{G})&=\dfrac{1}{2}\intg |u_n^\prime (x)|^2+W(x)|u_n|^2 \dif x-\dfrac{1}{p}\intg |u_n(x)|^p \dif x\\
		&\geq \dfrac{1}{2}\intg |u_n^\prime (x)|^2+W(x)|u_n|^2 \dif x-C\|u_n\|^{\frac{p}{2}+1}_{L^2(\mc{G})}\cdot \|u_n^\prime\|^{\frac{p}{2}-1}_{L^2(\mc{G})}\\
		&\geq \dfrac{1}{2}\|u_n^\prime\|^2_{L^2(\mc{G})}+\dfrac{W_{-}}{2}\mu-C\|u_n\|^{\frac{p}{2}+1}_{L^2(\mc{G})}\cdot \|u_n^\prime\|^{\frac{p}{2}-1}_{L^2(\mc{G})}\\
		&=\dfrac{1}{2}\|u_n^\prime\|^2_{L^2(\mc{G})}+\dfrac{W_{-}}{2}\mu-C\mu^{\frac{p}{4}+\frac{1}{2}}\|u^\prime_n\|^{\frac{p}{2}-1}_{L^2(\mc{G})}.
	\end{align*}
	Since $p\in (2,6)$, we obtain $\{u^\prime_n\}$ is bounded in $L^2(\mc{G})$. On the other hand, by $(W_1)$ and $(W_2)$, we see 
	\[\min\{1, W_{-}\}\mc{E}_{\mathrm{NLS}}(u_n,\mc{G}) \leq \|u_n\|^2_{H^1(\mc{G})}\leq \max\{ 1,W^+\} \mc{E}_{\mathrm{NLS}}(u_n,\mc{G}).\]
	Thus, the sequence $\{u_n\}$ is bounded in $H^1(\mc{G})$ and then weakly compact by Banach-Alaoglu's Theorem.	
\end{proof}

\begin{lemma}\label{lem6}
	Let $u\in\mathbb{H}$ and $u_e\geq 0$ for all $e\in E$. If $\mathrm{card}(u^{-1}(t))\geq N$ for almost every $t$ in the range of $u$, more precisely,
	\begin{equation}\label{eq17}
		\mathrm{card}\{x\in \mc{G}:\; u(x)=t \}\geq N\quad \text{for all}\quad a.e. \;t\in (0,\|u\|_{L^\infty(\mc{G})}), 
	\end{equation}
	then we have
	\begin{equation}\label{eq16}
		\mc{E}(u,\mc{G})\geq -\theta \left(\dfrac{2}{N}\right)^{2\be} \|u\|^{4\be+2}_{L^2(\mc{G})}.
	\end{equation}
\end{lemma}
\begin{proof}
	Let $u^\star\in H^1(\r^+)$ denotes the decreasing rearrangement of $u\geq 0$ on the positive halfline. It follows from \cite{AST2015} that 
	\begin{equation}\label{eq14}
	\|u^\star\|_{L^r(\r^+)}=\|u\|_{L^r(\mc{G})}\quad \text{for all}\;\;r\geq 1
	\end{equation} 
	and 
	\begin{equation}\label{eq15}
	\|(u^\star)^\prime \|_{L^2(\r^+)}\leq \|u^\prime \|_{L^2(\mc{G})}.
	\end{equation}
	Since \eqref{eq17} holds, it follows from \cite{Duff1970} that the inequality \eqref{eq15} has a more stronger form
	\begin{equation}\label{eq18}
		\|(u^\star)^\prime \|_{L^2(\r^+)}\leq \dfrac{1}{N^2}\|u^\prime \|_{L^2(\mc{G})}.
	\end{equation}
	Let $v(x)=u^{\star}(Nx)\in H^1(\r^+)$. Combining \eqref{eq14} and \eqref{eq18}, we have
	\begin{equation}\label{eq19}
		\|v\|^r_{L^r(\r^+)}=\dfrac{1}{N}\|u^\star\|^r_{L^r(\r^+)}= \dfrac{1}{N}\|u\|^r_{L^r(\mc{G})}\quad \text{for all}\quad r\geq 1
	\end{equation}
	and 
	\begin{equation}\label{eq21}
		\|v^\prime\|^2_{L^2(\r^+)}=N\|(u^\star)^\prime \|^2_{L^2(\r^+)}\leq \dfrac{1}{N}\|u^\prime \|^2_{L^2(\mc{G})}.
	\end{equation}
	In particular, $v\in S_{\frac{1}{N}\|u\|^2_{L^2(\mc{G})}}(\r^+)$.
	By \eqref{eq12} of Lemma \ref{lem2}, we see
	\begin{equation}\label{eq20}
		\mc{E}_{\mathrm{NLS}}(v,\r^+)\geq -\theta 2^{2\be}\left(\dfrac{1}{N} \right)^{2\be+1}\|u\|^{4\be+2}_{L^2(\mc{G})}.
	\end{equation}
	Therefore, by \eqref{eq19}, \eqref{eq21}, \eqref{eq20} and $(W_2)$, we obtain
	\begin{align*}
		\mc{E}(u,\mc{G})
		&\geq N\mc{E}_{\mathrm{NLS}}(v,\r^+)
		\geq -\theta \left(\dfrac{2}{N}\right)^{2\be} \|u\|^{4\be+2}_{L^2(\mc{G})}.
	\end{align*}
	This  complete the proof.
\end{proof}

	Similar to the concentration-compactness lemma in \cite[Lemma I.1]{Lions84}, we obtain the following version on metric graphs. Similar results can also be found in reference \cite{AST2019}. It is worth noting, however, that the potential function considered here is non‑vanishing.
	
\begin{proposition}\label{pro1}
	Let $\{u_n\}\subset S_\mu(\mc{G})$ be a sequence such that $u_n\rightharpoonup u$ in $H^1(\mc{G})$ and $u_n\to u$ a.e. on $\mc{G}$. Let
	\[\al:=\mu-\|u\|^2_{L^2(\mc{G})}\in [0,\mu]\] 
	be the loss of mass in the limit. Then there exists a subsequence of $\{u_n\}$ (still denote $\{u_n\}$) such that one of the following alternatives holds:
	\begin{itemize}
		\item[(\rmnum{1})] (vanishing) $\al=\mu$. Then $u= 0$ a.e. on $\mc{G}$ and $\liminf\limits_{n\to\infty}\mc{E}(u_n,\mc{G})\geq -\theta \mu^{2\be+1};$
		\item[(\rmnum{2})] (dichotomy) $\al\in (0,\mu)$. Then 
		\begin{equation}\label{eq26}
			\liminf\limits_{n\to\infty}\mc{E}(u_n,\mc{G})>\min \{-\theta \mu \al^{2\be}, \mc{E}(w,\mc{G})\},
		\end{equation}
		where 
		\[w(x):=\sqrt{\dfrac{\mu}{\mu-\al}}u(x)\quad \text{for all}\quad x\in \mc{G}. \]
		\item[(\rmnum{3})] (compactness) $\al=0$. Then $u\in S_\mu(\mc{G})$, $u_n\to u$ in $H^1(\mc{G})\cap L^p(\mc{G})$ and $\mc{E}(u,\mc{G})\leq \liminf\limits_{n\to\infty}\mc{E}(u_n,\mc{G})$.
	\end{itemize}
\end{proposition}

\begin{proof}
	{\bf(Vanishing case happen)}
	By $\al=\mu$, we have $u=0$ a.e. on $\mc{G}$ and $u_n\to 0$ strongly in $L^\infty_{loc}(\mc{G})$. 
	Since $\mc{E}(u_n,\mc{G})=\mc{E}(|u_n|,\mc{G})$, we may assume that $u_n\geq 0$. Let $M_n:=\max_{x\in \mc{G}}u_n(x).$ We check $\liminf\limits_{n\to\infty}\mc{E}(u_n,\mc{G})\geq -\theta \mu^{2\be+1}$ by considering the following two cases.
	
	{\bf Case 1.} There exists $e_0\in E$ such that $M_n$ is attained in $e_0$ where $e_0\sim I_{e_0}=[0,+\infty)$ is a halfline. In this case, we set $\delta_n:=u_n(0)$ and define the modified function $\widetilde{u}_n$ on $e_0$ by $\widetilde{u}_n(x)=|x-\delta_n|$ on $x\in [0,2\delta_n]$ and $\widetilde{u}_n(x)=u_n(x-2\delta_n)$ on $x\in (2\delta_n,+\infty)$. Thus, $\widetilde{u}_n\in H^1(e_0)$ and $\widetilde{u}_n(0)=u_n(0)=\delta_n$. Let us define 
	\begin{align}\label{eq23}
	\widetilde{u}_n(x):=\left\{\begin{array}{ll}
	|x-\delta_n|&\quad \text{if} \quad x\in [0,2\delta_n] \\
	u_n(x-2\delta_n) &\quad \text{if} \quad x\in (2\delta_n,+\infty) \\
	u_n(x)&\quad \text{if} \quad x\in \mc{G}\setminus {e_0}.
	\end{array} \right.
	\end{align}
	 Therefore $\widetilde{u}_n\in\mathbb{H}$ and $M_n=\max_{x\in e_0}\widetilde{u}_n(x)$. On the other hand, it follows from \eqref{eq23} that $\widetilde{u}_n(\delta_n)=0$ and $\widetilde{u}_n(x)=u_n(x-2\delta_n)\to 0$ as $x\to +\infty$. So
	\begin{equation*}
	\mathrm{card}\{x\in \mc{G}:\; \widetilde{u}_n=t \}\geq 2\quad \text{for all}\quad a.e. \;t\in (0,M_n).
	\end{equation*}
	We conclude from \eqref{eq16} of Lemma \ref{lem6} that 
	\begin{equation}\label{eq24}
	\mc{E}(\widetilde{u}_n,\mc{G})\geq -\theta  \|\widetilde{u}_n\|^{4\be+2}_{L^2(\mc{G})}.
	\end{equation}
	Since  $u_n\to 0$ in $L^\infty_{loc}(\mc{G})$, we have $u_n(0)=\widetilde{u}_n(0)=\delta_n\to 0$ as $n\to\infty$.
	Therefore, we deduce that 
	\begin{align*}
	\|\widetilde{u}_n\|^{2}_{L^2(\mc{G})}&
	=\int_{\mc{G}\setminus e_0}|\widetilde{u}_n|^2\dif x+\int_{I_{e_0}}|\widetilde{u}_n|^2\dif x\\
	&=\int_{\mc{G}\setminus e_0}|u_n|^2\dif x+\int_{0}^{2\delta_n}|x-\delta_n|^2\dif x+\int_{2\delta_n}^{+\infty}|u_n(x-2\delta_n)|^2\dif x\\
	&=m+o(1).
	\end{align*}
	
	Let $\va>0$ be an arbitrary, but fixed small number. By $(W_1)$ and $\delta_n\to 0$, there exists $R>0$ such that $|W(x+\delta_n)-W(x)|<\va$ for all $x>R$.
	Combining this, $\|u_n\|^2_{L^2(\mc{G})}\to m$, and $u_n\to 0$ in $L^\infty_{loc}(\mc{G})$, we have 
	\begin{align*}
	\left| \int_{0}^{+\infty}(W(x+\delta_n)-W(x))|u_n|^2\dif x\right| &\leq \int_{0}^{R}\left|W(x+\delta_n)-W(x) \right||u_n|^2\dif x\\
	&\quad +\int_{R}^{+\infty}\left|W(x+\delta_n)-W(x) \right||u_n|^2\dif x\\
	&\leq 2\max_{0\leq x\leq R}|W(x)|\int_{0}^{R}|u_n|^2\dif x+2\va  \int_{R}^{+\infty}|u_n|^2\dif x \\
	&=o(1).
	\end{align*}
	Thus 
	\begin{align*}
	\mc{E}(\widetilde{u}_n,\mc{G})&=\mc{E}(\widetilde{u}_n,\mc{G}\setminus e_0)+\mc{E}(\widetilde{u}_n,e_0)\\
	&=\mc{E}(u_n,\mc{G}\setminus e_0)+\dfrac{1}{2}\int_{0}^{2\delta_n}\left(|x-\delta_n|^\prime\right)^2\dif x+\dfrac{1}{2}\int_{0}^{2\delta_n}W(x)|x-\delta_n|^2\dif x-\dfrac{1}{p}\int_{0}^{2\delta_n}|x-\delta_n|^p\dif x\\
	&\quad +\dfrac{1}{2}\int_{2\delta_n}^{\infty}|u^\prime_n(x-2\delta_n)|^2\dif x+ \dfrac{1}{2}\int_{2\delta_n}^{+\infty}W(x)|u_n(x-2\delta_n)|^2\dif x-\dfrac{1}{p}\int_{2\delta_n}^{+\infty}|u_n(x-2\delta_n)|^p\dif x\\
	&=\mc{E}(u_n,\mc{G}\setminus e_0)+o(1)+\dfrac{1}{2}\int_{0}^{+\infty}|u^\prime_n|^2\dif x+\dfrac{1}{2}\int_{0}^{+\infty}W(x+\delta_n)|u_n|^2\dif x-\dfrac{1}{p}\int_{0}^{+\infty}|u_n|^p\dif x\\
	&=\mc{E}(u_n,\mc{G})+o(1)+\dfrac{1}{2}\int_{0}^{+\infty}(W(x+\delta_n)-W(x))|u_n|^2\dif x\\
	&=\mc{E}(u_n,\mc{G})+o(1).
	\end{align*}
	It follows from \eqref{eq24} that 
	\begin{align*}
	\liminf\limits_{n\to\infty}\mc{E}(u_n,\mc{G})=\liminf\limits_{n\to\infty}\mc{E}(\widetilde{u}_n,\mc{G})\geq \liminf\limits_{n\to\infty}(-\theta  \|\widetilde{u}_n\|^{4\be+2}_{L^2(\mc{G})} ) = -\theta \mu^{2\be+1}.
	\end{align*}
	{\bf Case 2.} $M_n$ is not attained in any $e\sim I_{e}=[0,+\infty)$. Thus, $M_n$ is attained in a complementary set. According to $u_n\to 0$ in $L^\infty_{loc}(\mc{G})$, we obtain $M_n\to 0$ as $n\to+\infty.$ By $(W_1)$, we have 
	\begin{align*}
	\mc{E}(u_n,\mc{G})&=\dfrac{1}{2}\intg |u^\prime_n|^2\dif x+\dfrac{1}{2}\intg W(x)|u_n|^2\dif x-\dfrac{1}{p}\intg |u_n|^p\dif x\\
	&\geq \dfrac{1}{2}\int_{\bigcup\limits_{e\in E}e} W(x)|u_n|^2\dif x+\dfrac{1}{2}\int_{\mc{G}\setminus \bigcup\limits_{e\in E}e} W(x)|u_n|^2\dif x-\dfrac{1}{p}\intg |u_n|^p\dif x\\
	&\geq o(1)-\dfrac{1}{p}\intg |u_n|^p\dif x\\
	&\geq -\dfrac{(M_n)^{p-2}}{p}\intg |u_n|^2\dif x+o(1)\\
	&=-\dfrac{(M_n)^{p-2}}{p}\mu+o(1).
	\end{align*}
	Thus $\liminf\limits_{n\to\infty}\mc{E}(u_n,\mc{G})\geq 0>  -\theta \mu^{2\be+1}.$
	
	{\bf(Dichotomy case happen)}
	Using Brezis-Lieb's Lemma we have
	\begin{align*}
		&\intg|u^\prime_n|^2\dif x=\intg |u^\prime_n-u^\prime|^2\dif x+\intg |u^\prime|^2\dif x+o(1);\\
		&\intg|u_n|^p\dif x=\intg |u_n-u|^p\dif x+\intg |u|^p\dif x+o(1);\\
		&\intg W(x)|u_n|^2\dif x=\intg W(x)|u_n-u|^2\dif x+\intg W(x)|u|^2\dif x+o(1).
	\end{align*}
	Thus
	\begin{equation*}
		\mc{E}(u_n,\mc{G})=\mc{E}(u_n-u,\mc{G})+\mc{E}(u,\mc{G})+o(1).
	\end{equation*}
	Since $u_n\rightharpoonup u$ in $H^1(\mc{G})$, the sequence $u_n \rightharpoonup u$ in $L^2(\mc{G})$ too. So 
	\begin{equation}\label{eq25}
		\|u_n-u\|^2_{L^{2}(\mc{G})}=\|u_n\|^2_{L^{2}(\mc{G})}-\|u\|^2_{L^{2}(\mc{G})}+o(1)\to \al
	\end{equation}
	as $n\to\infty$. Combining \eqref{eq25} and \eqref{eq10}, we obtain
	$\|u_n-u\|^2_{L^{\infty}_{loc}(\mc{G})}\to 0$ as $n\to\infty$. By replacing $u_n$ with $u_n -u$ and using a proof similar to the vanishing case, we obtain
	\[	\liminf\limits_{n\to\infty}\mc{E}(u_n-u,\mc{G})\geq -\theta\al^{2\be+1}.\]
	Thus
	\begin{align*}
		\liminf\limits_{n\to\infty}\mc{E}(u_n,\mc{G})=\liminf\limits_{n\to\infty}\mc{E}(u_n-u,\mc{G})+\mc{E}(u,\mc{G})\geq -\theta \al^{2\be+1}+\mc{E}(u,\mc{G}).
	\end{align*}
	Next, we derive an estimation of $\mc{E}(u,\mc{G})$. By direct calculation, we have 
	\begin{align*}
		\mc{E}(u,\mc{G})&=\dfrac{1}{2}\intg |u^\prime|^2+W(x)|u|^2\dif x-\dfrac{1}{p}\intg |u|^p\dif x\\
		&= \dfrac{\mu-\al}{2\mu}\intg|w^\prime|^2+W(x)|w|^2\dif x-\dfrac{1}{p}\intg \left(\dfrac{\mu-\al}{\mu} \right)^{\frac{p}{2}}|w|^p\dif x\\
		&= \dfrac{\mu-\al}{\mu}\left(\dfrac{1}{2}\intg|w^\prime|^2+W(x)|w|^2\dif x-\dfrac{1}{p}\intg \left(\dfrac{\mu-\al}{\mu} \right)^{\frac{p}{2}-1}|w|^p\dif x \right)  \\
		&> \dfrac{\mu-\al}{\mu}\left(\dfrac{1}{2}\intg|w^\prime|^2+W(x)|w|^2\dif x-\dfrac{1}{p}\intg |w|^p\dif x \right) \\
		&=\left(1- \dfrac{\al}{\mu}	\right) \mc{E}(w,\mc{G}).
	\end{align*}
	Therefore
	\begin{align*}
		\liminf\limits_{n\to\infty}\mc{E}(u_n,\mc{G})&>-\theta \al^{2\be+1}+\left(1-\dfrac{\al}{\mu}\right) \mc{E}(w,\mc{G})\\
		&=\left(-\theta \mu \al^{2\be} \right) \left(\dfrac{\al}{\mu} \right) +\mc{E}(w,\mc{G})\left(1-\dfrac{\al}{\mu}\right)
	\end{align*}
	This clearly implies \eqref{eq26}.
	
	{\bf(Compactness)} Since $u_n\in S_{\mu}(\mc{G})$ and $\al=0$, we have $\|u_n\|_{L^2{(\mc{G})}}\to\|u\|_{L^2{(\mc{G})}}$. 
	So $u_n\to u$ in $L^2{(\mc{G})}$ and therefore $u\in S_{\mu}(\mc{G})$. Moreover, since $u_n\rightharpoonup u$ in $H^1(\mc{G})$, we obtain that  $\{u_n\}$ is bounded in $H^1(\mc{G})$ in the sense of subsequence. It follows from \eqref{eq9} of Lemma \ref{lem1} that 
	\begin{align*}
		\|u_n-u\|^p_{L^p(\mc{G})}&\leq C\|u_n-u\|^{\frac{p}{2}+1}_{L^2(\mc{G})}\cdot \|(u_n-u)^\prime\|^{\frac{p}{2}-1}_{L^2(\mc{G})}\\
		&=o(1)\|u_n-u\|^{\frac{p}{2}-1}_{H^1(\mc{G})}\\
		&\leq o(1)\left(\|u_n\|_{H^1(\mc{G})}+\|u\|_{H^1(\mc{G})} \right)^{\frac{p}{2}-1}\\
		&=o(1).
	\end{align*}
	Thus $u_n\to u$ in $H^1(\mc{G})\cap L^p(\mc{G})$. Finally, according to the semicontinuity of norm, we have
	\begin{align*}
		\liminf\limits_{n\to\infty}\mc{E}(u_n,\mc{G})
		&\geq \dfrac{1}{2}\intg |u^\prime|^2+W(x)|u|^2\dif x-\dfrac{1}{p}\intg |u|^p\dif x=\mc{E}(u,\mc{G}).
	\end{align*}
	 The proof is finished.
\end{proof}

\section{Proof of theorem \ref{thm1}}
In this section, unless the contrary is explicitly stated, we will always assume that $\mc{G}$ is a noncompact, connected metric graph and having at least one bounded edge. 
Let $e$ be a bounded edge connecting vertices $\mathrm{v_1}$ and $\mathrm{v_2}$ of the graph $\mc{G}$. Furthermore, let us assume without loss of generality that $\mathrm{deg}(\mathrm{v_1})\ne 2\quad \text{and}\quad \mathrm{deg}(\mathrm{v_2})\ne 2.$
Note that the cases where $\mathrm{deg}(\mathrm{v_1})=1$ and $\mathrm{deg}(\mathrm{v_2})=1$ cannot occur. If such a situation were to arise, $e$ would be an isolated edge in the graph $\mc{G}$, contradicting the assumption that $\mc{G}$ is a noncompact and connected metric graph. Therefore, if necessary by swapping $\mathrm{v_1}$ and $\mathrm{v_2}$, we may make the following assumption:
\begin{equation}\label{eq27}
	\mathrm{deg}(\mathrm{v_1})\ne 2\quad \text{and}\quad \mathrm{deg}(\mathrm{v_2})\ge 3.
\end{equation}
The edge $e$ is called an end-edge of $\mc{G}$ if $\mathrm{deg}(\mathrm{v_1})=1$.

\begin{lemma}\label{lem5}
	For $\va>0$, there exist $\mu_\va:=\mu(\va,|e|,p, W^+)$ and $\theta^\prime=\frac{\theta}{2}>0$ which depend on only $p$ such that 
	\begin{equation}\label{eq28}
		\inf_{v\in \mc{H}_{\mu}(\mc{G})}\mc{E}(v,\mc{G})\leq -\theta^\prime (1-\va)\mu^{2\be+1}
	\end{equation}
	for all $\mu>\mu_\va$, where $|e|$ denotes the length of $e\in E$. Moreover, if $e$ is an end-edge of $\mc{G}$, we have 
	\begin{equation}\label{eq29}
		\inf_{v\in \mc{H}_{\mu}(\mc{G})}\mc{E}(v,\mc{G})\leq -\theta^\prime (1-\va)2^{2\be}\mu^{2\be+1}.
	\end{equation}
\end{lemma}
\begin{proof}
	Let $\phi$ denote the \textit{soliton} of unitary mass centered at the origin. It follows from \eqref{eq11} of Lemma \ref{lem2} that $\mc{E}_{\mathrm{NLS}}(\phi,\r)=-\theta$. Hence, for any $\va>0$, by a standard density arguments there exists $\phi_\va\in H^1(\r)$ with compact support such that
	\begin{align}\label{eq30}
	\left\{\begin{array}{ll}
	\mc{E}_{\mathrm{NLS}}(\phi_\va,\r)\leq -(1-\va)\theta,\\
	\|\phi_\va\|^2_{L^2(\r)}=1.
	\end{array} \right.
	\end{align}
	For any $\mu>0$, we define the \textit{dilation scaling} by 
	\[v_\mu(x):=\mu^\al\phi_\va(\mu^\be x) \quad \text{for all}\quad  x\in \r, \]
	where $\al=\frac{2}{6-p}$ and $\be=\frac{p-2}{6-p}$. A direct computation shows that
	\begin{align*}
		\|v_\mu\|^2_{L^2(\r)}&=\int_{\r}|v_\mu|^2\dif x=\mu^{2\al}\int_{\r}|\phi_\va(\mu^\be x)|^2\dif x\\
		&=\mu^{2\al-\be}\int_{\r}\phi_\va(y)|^2\dif y\\
		&=\mu.
	\end{align*}
	It follows from $(W_2)$ and \eqref{eq30} that
	\begin{align*}
		\mc{E}(v_\mu,\r)&=\dfrac{1}{2}\int_{\r}|v^\prime_\mu(x)|^2+W(x)|v_\mu(x)|^2\dif x-\dfrac{1}{p}\int_{\r}|v_\mu(x)|^p\dif x\\
		&\leq\mc{E}_{\mathrm{NLS}}(\phi_\va,\r)\cdot \mu^{2\be+1}+\dfrac{W^+}{2}\mu\\
		&\leq -\theta^\prime (1-\va)\mu^{2\be+1}
	\end{align*}
	for all $\mu \geq\sqrt[2\beta]{\frac{W^+}{\theta (1-\va)}}$.
	Combining $e\in E$ is bounded, the definition of $v_\mu$ and $\phi_\va\in H^1(\r)$ has compact support, we have $|\mathrm{supp} v_\mu|<|e|$ when $\mu$ large enough. Thus, we can fit $v_\mu$ on $e$ such that $v_\mu(\mathrm{v_1})=v_\mu(\mathrm{v_2})=0$. Let 
	\begin{align*}
	\widetilde{v}_\mu(x):=\left\{\begin{array}{ll}
	v_\mu(x)&\quad \text{if} \quad x\in e \sim [0,|e|] \\
	0&\quad \text{if} \quad x\in \mc{G}\setminus {e}.
	\end{array} \right.
	\end{align*}
	Clearly, we have $\widetilde{v}_\mu\in \mc{H}_\mu(\mc{G})$. Therefore, we have 
	\[\inf_{v\in \mc{H}_{\mu}(\mc{G})}\mc{E}(v,\mc{G})\leq\mc{E}(\widetilde{v}_\mu,\mc{G})\leq -\theta^\prime (1-\va)\mu^{2\be+1}. \]
	This is precisely \eqref{eq28}.
	
	Let $e$ is an end-edge of $\mc{G}$. It follows from \eqref{eq27} that $e$ is attached to $\mc{G}$ at $\mathrm{v_2}$ and $\mathrm{v_1}$ is the tip of the edge $e$. Now, in this case, for any $\mu>0$, we define the \textit{dilation scaling} by 
	\[w_\mu(x):=(2\mu)^\al\phi_\va((2\mu)^\be x) \quad \text{for all}\quad  x\in \r. \]
	Combining \eqref{eq30} and the fact that $w_\mu(-x)=w_\mu(x)$, we see that 
	\begin{align*}
		\|w_\mu(x)\|^2_{L^2(\r^+)}&=\int_{0}^{+\infty}|w_\mu(x)|^2\dif x=\dfrac{1}{2}\int_{\r}|w_\mu(x)|^2\dif x\\
		&=\dfrac{(2\mu)^{2\al}}{2}\int_{\r}|\phi_\va((2\mu)^\be x)|^2\dif x\\
		&=\mu.
	\end{align*} 
    On the other hand, by $(W_1)$ and \eqref{eq30}, we obtain that 
	\begin{align*}
		\mc{E}(w_\mu,\r^+)
		&=\dfrac{1}{4}\int_{\r}|w^\prime_\mu(x)|^2\dif-\dfrac{1}{2p}\int_{\r}|w_\mu(x)|^p\dif x+\int_{\r^+}W(x)|w_\mu(x)|^2\dif x\\
		&\leq \dfrac{1}{2}(2\mu)^{2\be+1} \mc{E}_{\mathrm{NLS}}(\phi_\va,\r)+\dfrac{W^+}{2}\mu\\
		&\leq -\dfrac{1}{2}\theta (1-\va)(2\mu)^{2\be+1}+\dfrac{W^+}{2}\mu\\
		&\leq-\theta^\prime (1-\va)2^{2\be}\mu^{2\be+1}
	\end{align*}
	for all $\mu \geq\sqrt[2\beta]{\frac{W^+}{\theta (1-\va)4^\beta}}$. Letting $\mu>0$ large enough, we have $|\mathrm{supp} w_\mu|<|e|$. Thus, we can fit $w_\mu$ on $e$ such that $w_\mu(\mathrm{v_2})=0$. 
	Let 
	\begin{align*}
	\widetilde{w}_\mu(x):=\left\{\begin{array}{ll}
	w_\mu(x)&\quad \text{if} \quad  x\in e \sim [0,|e|]\\
	0&\quad \text{if} \quad x\in \mc{G}\setminus {e}.
	\end{array} \right.
	\end{align*}
	Clearly, we have $\widetilde{w}_\mu\in \mc{H}_\mu(\mc{G})$. Therefore, we have 
	\[\inf_{v\in \mc{H}_{\mu}(\mc{G})}\mc{E}(v,\mc{G})\leq\mc{E}(\widetilde{w}_\mu,\mc{G}) \leq -\theta^\prime (1-\va)2^{2\be}\mu^{2\be+1}. \]
	This is precisely \eqref{eq29}.
\end{proof}
\begin{proposition}\label{pro2}
	There exists a mass threshold $\bar{\mu}$ such that the minimization problem \eqref{eq8} has a minimizer $u\in \mc{H}_{\mu}(\mc{G})$ for all $\mu>\bar{\mu}$.
\end{proposition}	
\begin{proof}
	Let $\va\in(0,\frac{1}{2})$. It follows from Lemma \ref{lem5} that there exists $\mu_\va:=\mu(\va,|e|,p, W^+)$ such that \eqref{eq28} and \eqref{eq29} hold for all $\mu>\mu_\va$. Let $\{u_n\}\subset \mc{H}_{\mu}(\mc{G})$ be a minimizing sequence of $\mc{E}(v,\mc{G})$. 
	For all $\mu>\mu_\va$, it follows from \eqref{eq28} that
	\[\lim_{n\to\infty}\mc{E}(u_n,\mc{G})=\inf_{v\in \mc{H}_\mu(\mc{G})}\mc{E}(v,\mc{G})<-\frac{\theta^\prime}{2} \mu^{2\be+1}. \]
    Thus, for large enough $n$, we have
    \[\mc{E}(u_n,\mc{G})\leq -\frac{\theta^\prime}{2} \mu^{2\be+1}. \]
    This implies that
    \begin{equation}\label{eq31}
    	\dfrac{\theta^\prime}{2}\mu^{2\be+1}+\dfrac{1}{2}\intg |u^\prime_n(x)|^2\dif x+\dfrac{1}{2}\intg W(x)|u_n(x)|^2\dif x\leq \dfrac{1}{p}\intg |u_n(x)|^p\dif x.
    \end{equation}
    On the one hand, it follows from \eqref{eq31} and $u_n\in \mc{H}_{\mu}(\mc{G})$ that 
    \begin{align*}
    	\dfrac{p\theta^\prime}{2}\mu^{2\be+1}+\dfrac{p}{2}\intg W(x)|u_n(x)|^2\dif x&\leq \intg |u_n(x)|^p\dif x
    	\leq \|u_n\|^{p-2}_{L^\infty(\mc{G})}\|u_n\|^2_{L^2(\mc{G})}
    	=\mu \|u_n\|^{p-2}_{L^\infty(e)}.
    \end{align*}
	Combining this and $(W_2)$, we have 
	\begin{align*}
		\|u_n\|^{p-2}_{L^\infty(e)}&\geq \dfrac{p\theta^\prime}{2}\mu^{2\be}+\dfrac{p}{2\mu}\intg W(x)|u_n(x)|^2\dif x\\
		&\geq\dfrac{p\theta^\prime}{2}\mu^{2\be}+\dfrac{p}{2}W_{-}\\
		&\geq C\mu^{2\be},
	\end{align*}
	where $C:=C(p)>0$ depend only on $p$. Solving the above inequality, we have 
	\begin{equation*}
		\|u_n\|^2_{L^\infty(e)}\geq C^\prime \mu^{\frac{4}{6-p}}= C^\prime \mu^{\be+1},
	\end{equation*}
	where $C^\prime:=C(p)^{\frac{2}{p-2}}>0$.
	On the other hand, according to \eqref{eq31} and \eqref{eq9}, we obtain that 
	\begin{align*}
		\|u^\prime_n\|^2_{L^2(\mc{G})}&\leq -\theta^\prime \mu^{2\be+1}+\dfrac{2}{p}\intg |u_n|^p\dif x-\intg W(x)|u_n|^2\dif x\\
		&\leq -\theta^\prime \mu^{2\be+1}+\dfrac{2}{p}\|u_n\|^{\frac{p}{2}+1}_{L^2(\mc{G})}\cdot \|u^\prime_n\|^{\frac{p}{2}-1}_{L^2(\mc{G})}-W_{-}\|u_n\|^2_{L^2(\mc{G})}\\
		&=\dfrac{2}{p}\mu^{\frac{p}{4}+\frac{1}{2}}\cdot\|u^\prime_n\|^{\frac{p}{2}-1}_{L^2(\mc{G})}- W_{-}\mu-\theta^\prime \mu^{2\be+1}\\
		&\leq \dfrac{2}{p}\mu^{\frac{p}{4}+\frac{1}{2}}\cdot\|u^\prime_n\|^{\frac{p}{2}-1}_{L^2(\mc{G})}.
	\end{align*}
	This implies that 
	\[ \|u^\prime_n\|^2_{L^2(\mc{G})}\leq \left(\dfrac{2}{p} \right)^{\frac{4}{6-p}}\mu^{2\be+1}.  \]
	Hence $\{u_n\}$ is bounded in $H^1(\mc{G})$. Along a subsequence of $\{u_{n}\}$, we may assume that $u_n\rhp u$ in $H^{1}(\mc{G})$ and $a.e.$ on $\mc{G}$. Since $u_n\to u$ uniformly on the edge $e$, we have 
	\[0<C^\prime \mu^{\be+1}\leq \lim_{n\to\infty}\|u_n\|^2_{L^\infty(e)}=\|u\|^2_{L^\infty(e)}\leq \|u\|^2_{L^\infty(\mc{G})}. \]
	Thus $u\neq 0$. This implies that the vanishing case of Proposition \ref{pro1} would not happen.
		
	Let $m:=\mu-\|u\|^2_{L^2(\mc{G})}\in [0,\mu)$. According to \eqref{eq10} of Lemma \ref{lem1}, we have 
	\begin{align*}
		C^\prime \mu^{\be+1}&
		\leq\|u\|^2_{L^\infty(\mc{G})}
		\leq 2\|u\|_{L^2(\mc{G})}\|u^\prime \|_{L^2(\mc{G})}
		=2\sqrt{\mu-m}\liminf\limits_{n\to\infty}\|u^\prime_n\|_{L^2(\mc{G})}\\
		&\leq 2\sqrt{\mu-m}\left(\dfrac{2}{p} \right)^{\frac{2}{6-p}}\mu^{\be+\frac{1}{2}}.
	\end{align*}
	Therefore
	\[m\leq (1-C^{\prime\prime})\mu,\]
	where $C^{\prime\prime}:=C^{\prime\prime}(p)>0$. 
	Let us note that $\|u\|_{L^\infty(e)}\leq \|u\|_{L^\infty(\mc{G})}$ and 
	\[\|u\|_{L^\infty(\mc{G})}\leq \liminf\limits_{n\to\infty}\|u_n\|_{L^\infty(\mc{G})}= \liminf\limits_{n\to\infty}\|u_n\|_{L^\infty(e)}=\|u\| _{L^\infty(e)},\]
	we obtain $\|u\|_{L^\infty(\mc{G})}=\|u\| _{L^\infty(e)}.$ Thus $u\in\mc{H}_\mu(\mc{G})$. Denote 
	\[w(x)=\sqrt{\dfrac{\mu}{\mu-m}}u(x). \]
	By direct calculation, we have $w\in\mc{H}_\mu(\mc{G})$. Hence
	\[\liminf\limits_{n\to\infty}\mc{E}(u_n,\mc{G})=\inf_{v\in \mc{H}_\mu(\mc{G})}\mc{E}(v,\mc{G}) \leq \mc{E}(w,\mc{G}). \]
	This implies that 
	\[\liminf\limits_{n\to\infty}\mc{E}(u_n,\mc{G})>-\theta \mu m^{2\be}\]
	if the dichotomy case of Proposition \ref{pro1} happen. By Lemma \ref{lem5}, we have 
	\[-\theta^\prime (1-\va)\mu^{2\be+1}> -\theta \mu m^{2\be}\]
	for all $\mu>\mu_\va$. Thus
	\[\sqrt[2\be]{\dfrac{1-\va}{2}}\mu<m\leq (1-C^{\prime\prime})\mu.\]
	Letting $\va\to \frac{1}{2}$, we obtain a contradiction. This implies that the dichotomy case of Proposition \ref{pro1} is impossible. It follows from Proposition \ref{pro1} that $u_n\to u$ in $H^1(\mc{G})\cap L^p(\mc{G})$ and $$\mc{E}(u,\mc{G})= \liminf\limits_{n\to\infty}\mc{E}(u_n,\mc{G})=\inf_{v\in \mc{H}_\mu(\mc{G})}\mc{E}(v,\mc{G}).$$
	This means that $u\in \mc{H}_\mu(\mc{G})$ is a minimizer for all $\mu>\mu_\va$.
\end{proof}	
	Next, we claim that, for large enough $\mu$, the minimizer which obtained by Proposition \ref{pro2} achieves its maximum only on the edge $e$.
\begin{proposition}\label{pro3}
	Let $u\in \mc{H}_\mu(\mc{G})$ be a minimizer which obtained by Proposition \ref{pro2}. Then, for large enough $\mu$, we have 
	\begin{equation}\label{eq32}
		\|u\|_{L^\infty(e)}>\|u\|_{L^\infty(\mc{G}\setminus e)}.
	\end{equation}
\end{proposition}
\begin{proof}
	By replacing $u$ by $|u|$, we may assume that $u\geq 0$. Let $M:=\|u\|_{L^\infty(e)}$. It follows from $u\in \mc{H}_\mu(\mc{G})$ that $M\geq \|u\|_{L^\infty(\mc{G}\setminus e)}$. We argue by contradiction. If not, we have 
	\begin{equation}\label{eq33}
		M=\|u\|_{L^\infty(\mc{G}\setminus e)}
	\end{equation}
	for large enough $\mu$. Let 
	\[B:=\bigcup_{g\in E}g, \]
	where $g$ be a bounded edge of $\mc{G}$. We denote
	\begin{equation}\label{eq34}
		\delta:=\max_{g\in B}\min_{x\in g}u(x). 
	\end{equation}
	Clearly, we have $\delta\in [0,M]$. We claim that, for any $t\in (\delta,M)$,
	\begin{equation}\label{eq35}
	\mathrm{card}\{x\in \mc{G}:\; u(x)=t \}\geq 2.
	\end{equation}
	Moreover, if $e$ is not an end-edge of $\mc{G}$, i.e. $\mathrm{deg}(\mathrm{v_1})\geq 3$,  then for any $t\in (\delta,M)$, we have
	\begin{equation}\label{eq36}
	\mathrm{card}\{x\in \mc{G}:\; u(x)=t \}\geq 3.
	\end{equation}
	Indeed, it follows from \eqref{eq33} that the claim is trivial if $\delta=M$. Therefore in what follows we shall always assume, without loss of generality, that $\delta \in [0,M)$. 
	By the definition of $M$ and \eqref{eq33}, we see that there exists $f\neq e$ such that $M=\|u\|_{L^\infty(f)}$ (may be $f\cap e=\emptyset$ or $e$ and $f$ share a vertex). Thus, according to the mean value theorem, there exist $x_1\in e$ and $x_2\in f$ such that $u(x_1)=u(x_2)=t$ for any $t\in  (\delta,M)$ (see Figure \ref{fig:1}). This implies \eqref{eq35}. 
	\begin{figure}[hb]
		\centering
		\begin{minipage}[b]{.39\linewidth}
			\centering
\begin{tikzpicture}[x=0.75pt,y=0.75pt,yscale=-1,xscale=1]
\draw [color={rgb, 255:red, 0; green, 0; blue, 0 }  ,draw opacity=1 ]   (347.67,41) -- (212.33,41) (247.67,45) -- (247.67,37) ;
\draw [shift={(212.33,41)}, rotate = 180] [color={rgb, 255:red, 0; green, 0; blue, 0 }  ,draw opacity=1 ][fill={rgb, 255:red, 0; green, 0; blue, 0 }  ,fill opacity=1 ][line width=0.75]      (0, 0) circle [x radius= 3.35, y radius= 3.35]   ;
\draw [shift={(347.67,41)}, rotate = 180] [color={rgb, 255:red, 0; green, 0; blue, 0 }  ,draw opacity=1 ][fill={rgb, 255:red, 0; green, 0; blue, 0 }  ,fill opacity=1 ][line width=0.75]      (0, 0) circle [x radius= 3.35, y radius= 3.35]   ;
\draw  [dash pattern={on 4.5pt off 4.5pt}]  (179.33,41) -- (212.33,41) ;
\draw  [dash pattern={on 4.5pt off 4.5pt}]  (347.67,41) -- (379.33,41) ;
\draw (247.33,45.4) node [anchor=north west][inner sep=0.75pt] {$x_{1}$};
\draw (275.33,17.4) node [anchor=north west][inner sep=0.75pt]    {$e$};
\draw (214.33,44.4) node [anchor=north west][inner sep=0.75pt]    {$\mathrm{v}_{1}$};
\draw (349.67,44.4) node [anchor=north west][inner sep=0.75pt]    {$\mathrm{v}_{2}$};
\end{tikzpicture}
		\end{minipage}%
		\begin{minipage}[b]{.69\linewidth}
			\centering
\begin{tikzpicture}[x=0.75pt,y=0.75pt,yscale=-1,xscale=1]

\draw [color={rgb, 255:red, 0; green, 0; blue, 0 }  ,draw opacity=1 ]   (347.67,41) -- (212.33,41) (247.67,45) -- (247.67,37) ;
\draw [shift={(212.33,41)}, rotate = 180] [color={rgb, 255:red, 0; green, 0; blue, 0 }  ,draw opacity=1 ][fill={rgb, 255:red, 0; green, 0; blue, 0 }  ,fill opacity=1 ][line width=0.75]      (0, 0) circle [x radius= 3.35, y radius= 3.35]   ;
\draw [shift={(347.67,41)}, rotate = 180] [color={rgb, 255:red, 0; green, 0; blue, 0 }  ,draw opacity=1 ][fill={rgb, 255:red, 0; green, 0; blue, 0 }  ,fill opacity=1 ][line width=0.75]      (0, 0) circle [x radius= 3.35, y radius= 3.35]   ;
\draw  [dash pattern={on 4.5pt off 4.5pt}]  (179.33,41) -- (212.33,41) ;
\draw  [dash pattern={on 4.5pt off 4.5pt}]  (347.67,41) -- (379.33,41) ;
\draw  [dash pattern={on 4.5pt off 4.5pt}]  (170.33,89) -- (212.33,90) ;
\draw [color={rgb, 255:red, 0; green, 0; blue, 0 }  ,draw opacity=1 ]   (347.67,90) -- (212.33,90) (247.67,94) -- (247.67,86) ;
\draw [shift={(212.33,90)}, rotate = 180] [color={rgb, 255:red, 0; green, 0; blue, 0 }  ,draw opacity=1 ][fill={rgb, 255:red, 0; green, 0; blue, 0 }  ,fill opacity=1 ][line width=0.75]      (0, 0) circle [x radius= 3.35, y radius= 3.35]   ;

\draw (247.33,48.4) node [anchor=north west][inner sep=0.75pt]    {$x_{2}$};
\draw (275.33,20.4) node [anchor=north west][inner sep=0.75pt]    {$f$};
\draw (189,60) node [anchor=north west][inner sep=0.75pt]   [align=left] {or};
\draw (249.5,95.9) node [anchor=north west][inner sep=0.75pt]    {$x_{2}$};
\draw (266.33,67.4) node [anchor=north west][inner sep=0.75pt]    {$f$};
\draw (352.33,85.4) node [anchor=north west][inner sep=0.75pt]    {$\infty $};
\end{tikzpicture}
		\end{minipage}%
		\caption{$\mathrm{card}\{x\in \mc{G}:\; u(x)=t \}\geq 2$}
		\label{fig:1}
	\end{figure}

	Since $e$ is not an end-edge of $\mc{G}$, we may assume that $\mathrm{deg}(\mathrm{v_1})\geq 3$ and $\mathrm{deg}(\mathrm{v_2})\geq 3$, where $e\succ \mathrm{v_1}$ and $e\succ \mathrm{v_2}$. Choose $x_3,x_4\in e$ satisfying $u(x_3)=\delta$ and $u(x_4)=M$. By replacing $\mathrm{v_1}$ by $\mathrm{v_2}$ (if necessary), we may assume that, on the edge $e$, $x_3$ lies between the vertex $\mathrm{v}_1$ and $x_4$.
	\begin{figure}[hb]
		\centering
		\begin{minipage}[b]{.59\linewidth}
			\centering
			\includegraphics[height=3cm]{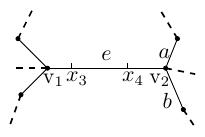}\vskip.3cm
		\end{minipage}
		\begin{minipage}[b]{.39\linewidth}
			\centering
			\includegraphics[height=4cm]{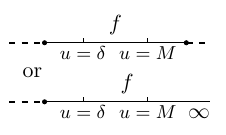}
		\end{minipage}
		\caption{$\mathrm{card}\{x\in \mc{G}:\; u(x)=t \}\geq 3$}
		\label{fig:2}
	\end{figure}
We assume, for notational convenience, that $x_1\in (x_3,x_4)$.
Let $a,b$ denote two edges adjacencing from $\mathrm{v}_2$ (see Figure \ref{fig:2}).

{\bf{Case 1}}: $f\neq a$.

If $u(\mathrm{v}_2)\leq t$, then there exists $x_5\in (x_4,|e|)$ such that $u(x_5)=t$ by the mean value theorem. Thus, for any $t\in (\delta,M)$, 
we have
\begin{equation*}
	\mathrm{card}\{x\in \mc{G}:\; u(x)=t \}\geq \mathrm{card}\{x_1,x_2,x_5 \}= 3,
\end{equation*}
where $x_1, x_5\in e$ and $x_2\in f$.

If $u(\mathrm{v}_2)> t$, by the continuity of $u$, after $u$ takes the variable $\mathrm{v}_2$, we discuss it in two cases. 
If the edge $a\in \mc{G}$ is bounded, according to \eqref{eq34}, we see that $(\delta, M)\subset u(a) $. 
Therefore, for any $t\in (\delta,M)$, there exists $x_5\in a$ such that $u(x_5)=t$ by the mean value theorem. 
Thus, for any $t\in (\delta,M)$, we have
\begin{equation}\label{eq37}
	\mathrm{card}\{x\in \mc{G}:\; u(x)=t \}\geq \mathrm{card}\{x_1,x_2,x_5 \}= 3,
\end{equation}
where $x_1\in e$, $x_2\in f$ and $x_5\in a$. On the other hand, if the edge $a\in \mc{G}$ is unbounded, by the fact $u(\mathrm{v}_2)>t$ and $u(x)\to 0$ as $x\to\infty$, then we duduce that there exists $x_5\in a$ such that $u(x_5)=t$. This implies that \eqref{eq37} holds.
	
{\bf{Case 2}}: $f\neq b$.

By arguments similar to those above, it is easy to show that \eqref{eq36} holds for this case.
	
{\bf{Case 3}}: $f= a=b$.

In this case, $f$ is necessarily a bounded edge and is a self-loop attached at $\mathrm{v}_2$ (see Figure \ref{fig:3}).
\begin{figure}[hb]
	\centering
	\includegraphics[height=3cm]{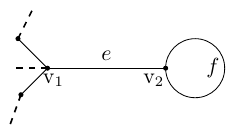}
	\caption{$\mathrm{card}\{x\in \mc{G}:\; u(x)=t \}\geq 3$}
	\label{fig:3}
\end{figure}
Since $u$ attains the values $\delta$ and $M$ on this loop, every intermediate value $t\in (\delta, M)$ is attained at least twice.
Thus, we also obtain \eqref{eq37}, where $x_1\in e$ and $x_2,x_5\in f$.	

Next, we want to apply the Lemma \ref{lem6}, and therefore we hope that \eqref{eq35} and \eqref{eq36} hold for all $t\in (0,M)$, not just for $t\in(\delta, M)$. 
Therefore, we construct a larger graph $\mc{G}^\prime$ from $G$ in the following way: given $\la > 0$, add two edges $e_1$ and $e_2$, both of length $\lambda$, to $G$, and connect them to $G$ via the vertex $x_3$ (see Figure \ref{fig:4}).
\begin{figure}[hb]
	\centering
	\includegraphics[height=3cm]{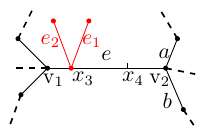}
	\caption{The graph $\mc{G}^\prime$ generated from $\mc{G}$}
	\label{fig:4}
\end{figure}
Putting a coordinate $s\in [0,\la]$	on each $e_i$ such that $e_i$ is attached to $x_3$ at $s=0$, where $i=1,2$. We extend $u$ to $\mc{G}^\prime$ by 
\begin{align}\label{Neq1}
	\widetilde{u}(x):=\left\{\begin{array}{ll}
	\dfrac{\delta}{\la} (\la-x)&\quad \text{if} \quad x\in e_i \sim [0,\la ]\\
	u(x)&\quad \text{if} \quad x\in \mc{G}\setminus {e_i},
	\end{array} \right.
\end{align}
where $i=1,2$. Thus, $\widetilde{u}\in H^1(\mc{G}^\prime)$.	Moreover, we have 
\[\mathrm{card}\{x\in \mc{G}^\prime:\; u(x)=t \}\geq 3 \qquad \text{for all} \quad t\in (0,\delta).\]
Indeed, due to $\mc{G}$ is a noncompact graph, there exists a unbounded edge $l\in \mc{G}^\prime$. 
According to $\widetilde{u}(x)=u(x)\to 0$ as $x\to\infty$, there exists $x_0\in l$ such that $\widetilde{u}(x_0)=t$ for any $t\in (0,\delta)$. Thus, $\widetilde{u}$ achieves every value $t\in (0,\delta)$ at least there times on $\mc{G}^\prime$: once on $e_1$, once on $e_2$, and once on $l$. 
Therefore, combining \eqref{eq35} and \eqref{eq36}, we deduce that
\begin{equation*}
	\mathrm{card}\{x\in \mc{G}^\prime:\; \widetilde{u}(x)=t \}\geq 2 \qquad \text{for all} \quad t\in (0,M).
\end{equation*}
	Moreover, if $e$ is not an end-edge of $\mc{G}^\prime$, i.e. $\mathrm{deg}(\mathrm{v_1})\geq 3$,  then we have
\begin{equation*}
	\mathrm{card}\{x\in \mc{G}^\prime:\; \widetilde{u}(x)=t \}\geq 3 \qquad \text{for all} \quad t\in (0,M).
\end{equation*}
According to Lemma \ref{lem6} ($N=2$), we see 
\begin{align*}
	\mc{E}(\widetilde{u},\mc{G}^\prime)&\geq -\theta \left(\int_{\mc{G}^\prime}|\widetilde{u}(x)|^2\dif x\right)^{2\be +1}\\
	&=-\theta \left(\intg |u(x)|^2\dif x +2\int_{0}^{\lambda}|u(s)|^2\dif s\right)^{2\be+1} \\
	&=-\theta \left(\mu +\dfrac{2}{3}\delta ^2\lambda \right)^{2\be+1}\\
	&>-\theta \left(\mu+\lambda \delta^2\right)^{2\be+1}.
\end{align*}
On the other hand, it follows from \eqref{Neq1} that 
\begin{align*}
	\mc{E}(\widetilde{u},\mc{G}^\prime)
	&=\mc{E}(u,\mc{G})+2\mc{E}(\dfrac{\delta}{\la} (\la-x),e_1)
	\leq \mc{E}(u,\mc{G})+\dfrac{\delta^2}{\lambda}+\lambda\delta^2W^+.
\end{align*}
This implies 
\begin{align}\label{eq38}
	\mc{E}(u,\mc{G})&\geq \mc{E}(\widetilde{u},\mc{G}^\prime)-\dfrac{\delta^2}{\lambda}-\lambda\delta^2W^+
	\geq -\theta \left(\mu+\lambda \delta^2\right)^{2\be+1}-\dfrac{\delta^2}{\lambda}-\lambda\delta^2W^+.
\end{align}
 According to Lemma \ref{lem6} ($N=3$), by arguments similar to those above, it is easy to show that 
 \begin{align}\label{eq39}
	\mc{E}(u,\mc{G})
	&\geq \mc{E}(\widetilde{u},\mc{G}^\prime)-\dfrac{\delta^2}{\lambda}-\lambda\delta^2W^+\nonumber\\
	&\geq -\theta \left(\dfrac{2}{3}\right)^{2\be}\left(\int_{\mc{G}^\prime}|\widetilde{u}(x)|^2\dif x\right)^{2\be +1}-\dfrac{\delta^2}{\lambda}-\lambda\delta^2W^+\nonumber\\
	&= -\theta \left(\dfrac{2}{3}\right)^{2\be}\left(\mu +\dfrac{2}{3}\delta ^2\lambda \right)^{2\be +1}-\dfrac{\delta^2}{\lambda}-\lambda\delta^2W^+,
 \end{align}
 when $\mathrm{deg}(\mathrm{v_1})\geq 3$. Now, let $\va>0$ small enough such that $\mu\geq \mu_\va$. It follows from \eqref{eq29} of Lemma \ref{lem5} and \eqref{eq38} that 
 \begin{equation*}
	\theta^{\prime} (1-\va)2^{2\be}\mu^{2\be+1}\leq -\inf_{v\in \mc{H}_\mu(\mc{G})}\mc{E}(v,\mc{G})\leq\theta \left(\mu+\lambda \delta^2\right)^{2\be+1}+\dfrac{\delta^2}{\lambda}+\lambda\delta^2W^+
 \end{equation*}
 for $\mathrm{deg}(\mathrm{v_1})=1$. Similarly, it follows from \eqref{eq28} of Lemma \ref{lem5} and \eqref{eq39} that 
 \begin{align*}
	\theta^\prime (1-\va)\mu^{2\be+1}&\leq -\inf_{v\in \mc{H}_\mu(\mc{G})}\mc{E}(v,\mc{G})\\
	&\leq \theta \left(\dfrac{2}{3}\right)^{2\be}\left(\mu +\dfrac{2}{3}\delta ^2\lambda \right)^{2\be +1}+\dfrac{\delta^2}{\lambda}+\lambda\delta^2W^+\\
	&\leq \theta \left(\dfrac{2}{3}\right)^{2\be}\left(\mu +\lambda\delta^2 \right)^{2\be +1}+\dfrac{\delta^2}{\lambda}+\lambda\delta^2W^+
 \end{align*}
 for $\mathrm{deg}(\mathrm{v_1})\geq 3$. Note that for any bounded edge $h\in \mc{G}^\prime$, we have 
 \begin{equation*}
	\left(\min_{x\in h}u(x)\right)^2\leq\dfrac{1}{|h|}\int_{h}|u(x)|^2\dif x\leq \dfrac{\mu}{l_0},
 \end{equation*}
 where $l_0>0$ is the length of the shortest edge of $\mc{G}^\prime$. Thus, it follows from \eqref{eq34} that 
 \begin{equation*}
 	\delta^2\leq \dfrac{\mu}{l_0}.
 \end{equation*}
Hence, if $\mathrm{deg}(\mathrm{v_1})=1$, we have
\begin{align*}
	\theta^\prime (1-\va)2^{2\be}\mu^{2\be+1}
	&\leq \theta \left(\mu+\lambda \delta^2\right)^{2\be+1}+\dfrac{\delta^2}{\lambda}+\lambda\delta^2W^+\\
	&\leq \theta\left(\mu+\dfrac{\mu}{l_0}\lambda \right)^{2\be+1}+\dfrac{\mu}{\lambda l_0}+\lambda\dfrac{\mu}{l_0}W^+.
\end{align*}
This implies 
\begin{equation}\label{eq40}
	\theta^\prime (1-\va)2^{2\be}\mu^{2\be}\leq \theta \left(1+\dfrac{\lambda}{l_0}\right)^{2\be+1}\mu^{2\be}+\dfrac{1}{\lambda l_0}+\dfrac{\lambda W^+}{l_0}.
\end{equation}
Fix $\lambda>0$ small enough. Let $\va>0$ small enough, such that the coefficient of $\mu^{2\be}$ on the left of \eqref{eq40} is strictly bigger than the corresponding coefficient on the right. Thus, a contradiction is obtained if $\mathrm{deg}(\mathrm{v_1})=1$ and $\mu$ large enough. Similarly, if $\mathrm{deg}(\mathrm{v_1})\geq 3$, we have 
\begin{align*}
	\theta^\prime (1-\va)\mu^{2\be+1}&\leq\theta \left(\dfrac{2}{3}\right)^{2\be}\left(\mu +\lambda\delta^2 \right)^{2\be +1}+\dfrac{\delta^2}{\lambda}+\lambda\delta^2W^+\\
	&\leq \theta \left(\dfrac{2}{3}\right)^{2\be}\left(\mu +\dfrac{\mu}{l_0}\lambda\right)^{2\be +1}+\dfrac{\mu}{\lambda l_{0}}+\lambda\dfrac{\mu}{l_0}W^+.
\end{align*}
This implies
\begin{equation}\label{eq41}
	\theta^\prime (1-\va)\mu^{2\be}\leq \theta \left(\dfrac{2}{3}\right)^{2\be}\left(1 +\dfrac{\lambda}{l_0}\right)^{2\be +1}\mu^{2\be}+\dfrac{1}{\lambda l_{0}}+\dfrac{\lambda W^+}{l_0}.
\end{equation}
Fix $\lambda>0$ small enough. Let $\va>0$ small enough, such that the coefficient of $\mu^{2\be}$ on the left of \eqref{eq41} is strictly bigger than the corresponding coefficient on the right. Thus, a contradiction is obtained if $\mathrm{deg}(\mathrm{v_1})\geq3$ and $\mu$ large enough. Thus, the result follows.

\end{proof}
\begin{proof}[{\bf{Proof of Theorem \ref{thm1}}}]
	Let $u\in \mc{H}_\mu(\mc{G})$ be a minimizer which obtained by Proposition \ref{pro2}. Clearly, $u\in S_\mu(\mc{G})$. According to the strict inequality of Proposition \ref{pro3}, \eqref{eq32} is stable under small perturbations of $u$ in the $L^\infty$ norm.
	Indeed, according to \eqref{eq32}, we have $\delta:=\|u\|_{L^\infty(e)}-\|u\|_{L^\infty(\mc{G}\setminus e)}>0.$ Let $\va>0$ small enough and $v:=u+w$, where $\|v-u\|_{L^{\infty}(\mc{G})}=\|w\|_{L^{\infty}(\mc{G})}<\va.$ Therefore 
	\begin{align*}
		&\|v\|_{L^\infty(e)}\geq \|u\|_{L^\infty(e)}-\|w\|_{L^\infty(e)}\geq \|u\|_{L^\infty(e)}-\va\\
		&\|v\|_{L^\infty(\mc{G}\setminus e)}\leq \|u\|_{L^\infty(\mc{G}\setminus e)}+\|w\|_{L^\infty(\mc{G}\setminus e)}\leq \|u\|_{L^\infty(\mc{G}\setminus e)}+\va.
	\end{align*}
	Letting $\va\in (0,\frac{\delta}{2})$, we have 
	\[\|v\|_{L^\infty(e)}-\|v\|_{L^\infty(\mc{G}\setminus e)}>0. \]
	On the other hand, since $H^1(\mc{G})\hookrightarrow L^\infty(\mc{G})$, we obtain that \eqref{eq32} is also stable under small perturbations of $u$ in the $H^1(\mc{G})$ norm. 
	Recalling the definition of $\mc{H}_\mu(\mc{G})$, we obtain that $u$ lies in the interior of $\mc{H}_\mu(\mc{G})$. 
	This implies that $u$ is not only a global minimizer in $\mc{H}_\mu(\mc{G})$, but also a local minimizer in $S_\mu(\mc{G})$.
	Thus, a standard argument yields that $u$ is a nontrivial solution of the equation \eqref{eq1}.
	It follows from Proposition \ref{pro3} that the maximum of $u$ is attained only at the bounded edge $e$.
	Next, we claim that either $u>0$ or $u<0$. Clearly, since $u$ is a minimizer, $|u|$ is also a minimizer. Since $|u|\in S_\mu(\mc{G})$, $|u|$ is not identically zero. 
	According to Proposition 3.3 of \cite{AST2015}, up to phase multiplication, one has that $|u|>0$ on $\mc{G}$. Since $\mc{G}$ is connected, we obtain either $u>0$ or $u<0$.
	
	Recalling the definition of $\mc{H}(\mc{G})$, the edge $e$ can be chose in $k$ different ways, so one can obtain $k$ local minimizer in $S_\mu(\mc{G})$. In other words, there exists at least $k$ bounded states $\{u_i\}_{i=1}^{k}$ for large $\mu$. Moreover, the maximum of $u_i$ is attained only at the edge $e_i$ and either $u_i>0$ or $u_i<0$.
	Finally, for each bounded edge $e_i$, by \eqref{eq32} of Proposition \ref{pro3}, we obtain that those bounded states are geometrically distinct solutions.
	
\end{proof}

\bigskip


\noindent{\bf Acknowledgments:} The author thanks the anonymous reviewers for their valuable comments.

\noindent{\bf Research ethics:} Not applicable.	



\noindent{\bf Use of Large Language Models, AI and Machine Learning Tools:} None declared.

\noindent{\bf Conflict of interest:} 
The author states no conflict of interest.


\noindent{\bf Research funding:} 
Liu is supported by Research Start-up Fund of Gannan Normal University (No.415273).

\noindent{\bf Data availability:} Not applicable.




\begin{thebibliography}{99}
	{\footnotesize
	
	
	\bibitem{ACFN2014} R. Adami, C. Cacciapuoti, D. Finco, D. Noja, Constrained energy minimization and orbital stability for the NLS equation on a star graph, Ann. Inst. H. Poincar\'{e} Anal. Non Lin\'{e}aire, 31(2014) 1289--1310.
	
	
	\bibitem{ACFN2016} R. Adami, C. Cacciapuoti, D. Finco, D. Noja, Stable standing waves for a NLS on star graphs as local minimizers of the constrained energy, J. Differ. Equ., 260(2016) 7397--7415.
	
	
	\bibitem{AST2015} R. Adami, E. Serra, P. Tilli, NLS ground states on graphs, Calc. Var. Partial Differential Equations, 54(2015) 743--761.
	
	\bibitem{AST2016} R. Adami, E. Serra, P. Tilli, Threshold phenomena and existence results for NLS ground states on metric graphs, J. Funct. Anal. 271 (2016) no. 1, pp. 201--223.
	
	\bibitem{ASTcmp} R. Adami, E. Serra, P. Tilli, Negative energy ground states for the $L^2$-critical NLSE on metric graphs. Comm. Math. Phys. 352(2017), no. 1, 387--406.
	
	\bibitem{AST2019} R. Adami, E. Serra, P. Tilla, Multiple positive bound states for the subcritical NLS equation on metric graphs. Calc. Var. Partial Differential Equations, 58(2019) 5. 
	
	\bibitem{ACT2024} F. Agostinho, S. Correia, H. Tavares, Classification and stability of positive solutions to the NLS equation on the $\mathcal{T}-$ metric graph. Nonlinearity, 37(2024) 025005.
	
	\bibitem{BL} H. Berestycki, P. Lions, Nonlinear scalar field equations, I existence of a ground state. Arch. Rational Mech. Anal., 82(1983), 313--345. 
	
	\bibitem{BCFK} G. Berkolaiko, R. Carlson, S. Fulling, P. Kuchment, Quantum graphs and their applications, Contemporary mathematics, vol. 415. Providence, RI: American Mathematical Society, 2006.	
	
	\bibitem{BK} G. Berkolaiko, P. Kuchment, Introduction to quantum graphs. In: Mathematical Surveys and Monographs, vol. 186. Providence, RI: American Mathematical Society, 2013.
	
	\bibitem{BMD2021} G. Berkolaiko, J. Marzuola, D. Pelinovsky, Edge-localized states on quantum graphs in the limit of large mass. Ann. Inst. H. Poincar\'{e} Anal. Non Lin\'{e}aire, 38(2021) 1295--1335.
	
	\bibitem{CDS} C. Cacciapuoti, S. Dovetta, E. Serra, Variational and stability properties of constant solutions to the NLS equation on compact metric graphs. Milan J. Math. 86(2018), no. 2, 305--327.
	
	\bibitem{CJS2022} X. Chang, L. Jeanjean, N. Soave, Normalized solutions of $L^2$–supercritical NLS equations on compact metric graphs. Ann. Inst. H. Poincar\'{e} Anal. Non Lin\'{e}aire, 41(2022) 933--959. 
	
	\bibitem{DDGS2023} C. De Coster, S. Dovetta, D. Galant, E. Serra, On the notion of ground state for nonlinear Schr\"odinger equations on metric graphs. Calc. Var. Partial Differential Equations, 62(2023) 159.
	
	
	\bibitem{DDGS2026} C. De Coster, S. Dovetta, D. Galant, E. Serra, C. Troestler, Constant sign and sign changing NLS ground states on noncompact metric graphs, Anal. PDE 19 (2) (2026) 203--240.
	
	\bibitem{D2018JDE} S. Dovetta, Existence of infinitely many stationary solutions of the $L^2$-subcritical and critical NLSE on compact metric graphs. J. Differential Equations 264 (2018), no. 7, 4806--4821.
	
	\bibitem{DST2020} S. Dovetta, E. Serra, P. Tilli, NLS ground states on metric trees: existence results and open questions. J. London Math. Soc. (2), 102(2020) 1223--1240.
	
	\bibitem{Duff1970} G. Duff, Integral inequalities for equimeasurable rearrangements, Canadian J. Math., 22(1970) 408--430.
	
	\bibitem{EKKST} P. Exner, J. Keating, P. Kuchment, T. Sunada, A. Teplyaev, Analysis on graphs and its applications, Proceedings of Symposia in Pure Mathematics, vol. 77. Providence, RI: American Mathematical Society, 2008.
	
	\bibitem{JS2026} L. Jeanjean, L. Song, Sign-changing prescribed mass solutions for $L^2$-supercritical NLS on compact metric graphs, J. Funct. Anal. 290(2026) 111422.
	
	\bibitem{Lions84} P. L. Lions, The concentration-compactness principle in the calculus of variations. The locally compact case \Rmnum{1}, Ann. Inst. H. Poincar\'{e} Anal. Non Lin\'{e}aire, vol. 1, no. 2, pp. 109--145, 1984. 
	
	\bibitem{Noja2014} D. Noja, Nonlinear Schr\"odinger equation on graphs: recent results and open problems, Philos. Trans. Roy. Soc. A, 372(2014) 20130002.
	
	\bibitem{P2018} A. Pankov, Nonlinear Schr\"odinger equations on periodic metric graphs. Discrete Contin. Dyn. Syst., 38(2018) 697--714.
	
	
	\bibitem{PS2022} D. Pierotti, N. Soave, Ground states for the NLS equation with combined nonlinearities on noncompact metric graphs. SIAM J. Math. Anal., 54(2022) 768--790.
	
	\bibitem{Ten2016} L. Tentarelli, NLS ground states on metric graphs with localized nonlinearities. J. Math. Anal. Appl., 433(2016) no. 1, pp. 291--304.
	

	
	

	
	}
\end{thebibliography}
\end{document}